\documentclass[12pt]{amsart}
\usepackage[top=1.5in, bottom=1.5in, left=1.25in, right=1.25in]	{geometry}

\usepackage{mathtools}
\mathtoolsset{showonlyrefs,showmanualtags}

\usepackage{hyperref} 
\hypersetup{
    colorlinks=true,       
    linkcolor=blue,          
    citecolor=magenta,        
    filecolor=magenta,      
    urlcolor=cyan           
}

\usepackage[backrefs]{amsrefs}

\usepackage[all]{xy}
\usepackage{amsmath}
\usepackage{amssymb}
\usepackage{amsthm}
\usepackage{amscd}

\newcommand{\EE}{\mathbb E}

\newcommand{\Z}{\mathbb Z}

\newcommand{\TT}{\mathbb{ T} }

\newcommand{\Ca}{\mathcal{C}}

\newtheorem{conjecture}[equation]{Conjecture}
\newtheorem{theorem}[equation]{Theorem}

\newtheorem{corollary}[equation]{Corollary}
\newtheorem{lemma}[equation]{Lemma}

\DeclareMathOperator{\sgn}{sgn}

\DeclareMathOperator{\T}{\mathcal{T}}

\numberwithin{equation}{section}

\usepackage{mathtools}
\mathtoolsset{showonlyrefs,showmanualtags} 

\usepackage{hyperref} 
\hypersetup{
    colorlinks=true,       
    linkcolor=blue,          
    citecolor=magenta,        
    filecolor=magenta,      
    urlcolor=cyan           
}

\begin{document}
\title[Discrete Carleson Theorems]{Sparse Bounds for  Random \\ Discrete  Carleson Theorems} 
\author[B. Krause]{Ben Krause}
\address{
Department of Mathematics
The University of British Columbia \\
1984 Mathematics Road
Vancouver, B.C.
Canada V6T 1Z2}
\email{benkrause@math.ubc.ca}
\author[M. T. Lacey]{Michael T. Lacey}
\address{
School of Mathematics
Georgia Institute of Technology \\
686 Cherry Street
Atlanta, GA 30332-0160}
\email{lacey@math.gatech.edu}

\thanks{B.K is an NSF Postdoctoral Research Fellow.  Research of M.L. supported in part by grant NSF-DMS 1265570 and  NSF-DMS-1600693.}
\date{\today}

\begin{abstract} 
We study discrete random variants of the Carleson maximal operator.  
Intriguingly, these questions remain subtle and difficult, even in this setting.  
Let  $\{X_m\}$ be an independent sequence of $\{0,1\}$ random variables with expectations 
\[ \mathbb E  X_m = \sigma_m = m^{-\alpha}, \ 0 < \alpha <  1/2,  \]
and  $ S_m = \sum_{k=1} ^{m} X_k$.  Then the maximal operator below   almost surely  is bounded from $ \ell ^{p}$ to $ \ell ^{p}$, 
provided the Minkowski dimension of $ \Lambda  \subset [-1/2, 1/2]$ is strictly less than $ 1- \alpha $.   
\begin{equation*}
\sup_{\lambda \in \Lambda  } \left| \sum_{m\neq 0} X _{\lvert  m\rvert } 
\frac{e(  \lambda m )}{   \textup{sgn} (m)S _{ |m| }} f(x- m) \right|. 
\end{equation*}
This operator also satisfies a sparse type bound.  
The form of the sparse bound immediately implies weighted estimates in all $ \ell ^{2}$, which are novel in this setting.  
Variants and extensions are also considered. 
\end{abstract}

\maketitle

\section{Introduction}
The Carleson maximal operator \cite{MR0199631} controls the pointwise convergence of Fourier series. 
In the discrete setting, this estimate is as follows.  

\begin{theorem}\label{disc} The discrete Carleson maximal operator
\begin{equation}\label{e:Carleson}
C f (x) := 
\sup_{0 \leq \lambda \leq 1} \left| \sum_{m \neq 0} f(n-m) \frac{e(\lambda m)}{m} \right| 
\end{equation}
is bounded on $\ell^p(\Z), \ 1<p<\infty$. Here and throughout, $e(t) := e^{2\pi i t}$.
\end{theorem}

The original article of Carleson addressed the  Theorem above  with the integers replaced by the  circle group, in the case of $ p=2$, with its extension to $ L ^{p}$ due to Hunt \cite{MR0238019}. 
It was transferred to the real line by Kenig and Tomas \cite{MR583403}, but the variant for the integers was not noticed for several years. We are aware of two independent references for the Theorem above, that of  Campbell and Petersen \cite {MR958884}*{Lemma 2} and Stein and Wainger  \cite {MR1056560}.

In its much more well known version on the real line, this Theorem has several  variants and extensions. 
For instance, the polynomial variant of Stein and Wainger \cite{MR1879821}, and its deep extension by Victor Lie \cites{MR2545246,11054504}.  
Pierce and Yung \cite{150503882} have recently established certain Radon transform versions of Carleson's Theorem. 
These are powerful and deep facts.

The discrete versions of these results has only recently been investigated. To give the flavor of results that are under consideration, we recall this conjecture of Lillian Pierce \cite{LP_PC}.  Below, and throughout this paper we set 
$ e (t) = e ^{2 \pi i t}$, and identify the fundamental domain for $ \mathbb T = \mathbb R / \mathbb Z $ as $ [-1/2, 1/2]$.  

\begin{conjecture}\label{j:}  The following inequality holds on $ \ell ^2 (\mathbb Z )$. 
\begin{equation*}
\Bigl\lVert  \sup _{-\frac 12  \leq  \lambda \leq \frac 12  } 
\Bigl\lvert  \sum_{n\neq 0}  f (x-n)  \frac {e (\lambda  n ^2 )} n  \Bigr\rvert\, 
\Bigr\rVert _{\ell ^2 } \lesssim \lVert f\rVert _{\ell  ^2 }. 
\end{equation*}
\end{conjecture}

A recent paper of the authors \cite{151206918}  supplied sufficient conditions on $ \Lambda $ 
so that if one forms a restricted supremum over $ \lambda \in  \Lambda $, the maximal function above would be bounded on $ \ell ^2 $.  
Even under this restriction, in which we require sufficiently small \emph{arithmetic Minkowski dimension},  our proof is difficult,  even for examples of $ \Lambda $ being a sequence that converge very rapidly to the origin.  
The interested reader is referred to \cites{151206918,160408695} for more background (including the definition of arithmetic Minkowski dimension), and related results.  

It is therefore of some interest to study random versions of these questions, in which  we expect 
some of the severe obstacles in the arithmetic case to be  of an easier nature.  
This is so,  but even still, we will not be able to prove the most natural conjectures, and indeed even find that the random versions still have remnants of the  arithmetic difficulties of the non-random versions.  

We consider two examples of random Carleson operators.   
From now on $\{ X_n \;:\; n\in \mathbb Z  \}$ will denote a sequence of independent $\{0,1\}$ random variables (on a probability space $\Omega$) with expectations
\begin{equation}\label{e:X}
\EE X_m := \sigma_m = m^{-\alpha}, \ 0 < \alpha < 1.
\end{equation}
Also define the partial sums by 
\begin{equation}\label{e:Sm}
S_ n = 
\begin{cases}
\sum_{m=1} ^{n} X_m  & n >0 
\\
- S _{-n}   & n < 0 
\end{cases}. 
\end{equation}
By the Law of Large Numbers, $S_n$ is approximately $ c_ \alpha n ^{1- \alpha }$.

In the first random examples, the analogy to the (linear) Carleson theorem stronger, since the frequency modulation and shift parameters agree. 
\begin{equation}\label{e:Tdef}
\mathcal{T}_{\alpha,\Lambda}^{\omega}f(x) := \sup_{\lambda \in \Lambda} \left| \sum_{m\neq 0} X _{\lvert  m\rvert } 
\frac{e(  \lambda m )}{  S _{ m }} f(x- m) \right|.
\end{equation}
We consider arbitrary $ 0< \alpha < 1 $ for the above operator, but 
in the second example below, we only consider $ \alpha = 1 - \frac{1}{d} $ where $ d\geq 3$ is an integer, and have distinct frequency modulations and shift parameters. 
\begin{gather}  \label{e:C}
\Ca_{\alpha,\Lambda}^\omega f(x) := 
 \sup_{\lambda \in \Lambda} \left| \sum_{m \neq 0}   X_m \frac{e( \lambda m )}{S _{  m }} f (x-S_m) \right|  . 
\end{gather}
Note that $ \lvert  S_m\rvert  \approx m ^{d}$ above is random version of monomial power.  
In both definitions, we are using the definition \eqref{e:Sm} to define $ S _{m}$ for negative $ m$.  

We will not be able to control the unrestricted supremum of $ \lambda $, using \emph{Minkowski dimension} as a sufficient condition for the boundedness of our maximal operators.  
Given $ \Lambda \subset [-1/2,1/2]$,  and $ 0< \delta < 1$, let  
 $N(\delta) = N_\Lambda(\delta)$  be the fewest number of  intervals $ I_1 ,\dotsc, I_N$ required to cover $\Lambda$, 
 subject to the condition at $ \lvert  I_n\rvert < \delta  $ for all $ 1\leq n \leq N$.  
We say that $ \Lambda $ has \emph{Minkowski dimension $ d$} if 
\begin{equation}\label{e:Mink}
 C_d := \sup_{0< \delta \leq 1} N(\delta) \delta^{d } < \infty . 
\end{equation}
The point of interest in the next theorem are that we (a) allow arbitrary $ 0< \alpha <1$, (b) have an explicit assumption on the Minkowski dimension of $ \Lambda $, and (c) obtain  \emph{a sparse bound} for the operator.  

\begin{theorem}\label{random0}
Suppose
\[ \EE X_m = \sigma_m = m^{-\alpha}, \qquad  0 < \alpha < 1, \]
and let $\Lambda \subset [0,1]$ have upper Minkowski dimension $ \delta $ strictly less than $    {1- \alpha }  $.
Then almost surely, these two properties hold. 
\begin{enumerate}
\item   For all $ 1 < p < \infty $,  we have $ \lVert T_{\alpha,\Lambda}^\omega  \;:\; \ell ^{p} \mapsto \ell ^{p}\rVert < \infty $. 

\item There is a $ 1 < r = r (\alpha ,\delta) < 2 $ so that for 
finitely supported functions $ f$ and $ g$, there is a sparse operator $ \Pi _{\mathcal S,r}$ so that 
\begin{equation} \label{e:Talpha_sparse}
\lvert  \langle \T_{\alpha,\Lambda}^\omega f, g   \rangle\rvert  \lesssim \Pi _{\mathcal S, r} (f,g). 
\end{equation}
\end{enumerate}
In particular, there holds almost surely, for all   weights $ w \in A_{2/r} \cap RH _{r/ (2-r)}$, 
\begin{equation}\label{e:A3}
 \lVert \T_{\alpha,\Lambda}^\omega  \;:\; \ell ^{2} ( \mathbb Z , w ) \mapsto \ell ^{2} (\mathbb Z , w)\rVert 
 < \infty . 
\end{equation}
\end{theorem}

In the second conclusion, we are using the notation of \S\ref{s:sparse}, specifically see \eqref{e:SP_def}.  It implies the  weighted inequalities \eqref{e:A3}, as is explained in  that section. In particular, there is a slightly wider class of inequalities that are true, as specified in \eqref{e:wts}.
We are \emph{not aware} of any prior weighted inequality for a discrete variant of the Carleson operators (except the Carleson operator itself).   (For discrete random Hilbert transforms, see \cite{lacey_spencer}.)
We remark that we could keep track of the dependence of the sparse index $ r$ as a function of $ \alpha $ and $ \delta $,
but we don't do so.  

\begin{theorem}\label{random}
Suppose $ d \geq 3$ is an integer and 
\[ \EE X_m = \sigma_m = m^{-\alpha}, \ \alpha = 1- \frac{1}{d}.  \]
Let $\Lambda \subset [-1/2,1/2]$ have  upper Minkowski dimension $ \delta $ strictly less than $ 1/d$, and  
$ \Lambda \cap (- \epsilon ,  \epsilon )=\emptyset $ for some $ 0 < \epsilon < \tfrac 14$. 
Then almost surely,  these two conclusions hold. 

\begin{enumerate}
\item   For all $ 1 < p < \infty $,  we have $ \lVert  \Ca_{\alpha,\Lambda}^\omega  \;:\; \ell ^{p} \mapsto \ell ^{p}\rVert < \infty $. 

\item There is a $ 1 < r = r (d,\delta) < 2 $ so that for 
finitely supported functions $ f$ and $ g$, there is a sparse operator $ \Pi _{\mathcal S,r}$ so that 
\begin{equation} \label{e:Talpha_sparse}
\lvert  \langle \Ca_{\alpha,\Lambda}^\omega f, g   \rangle\rvert  \lesssim \Pi _{\mathcal S, r} (f,g). 
\end{equation}
\end{enumerate}
In particular, the inequality \eqref{e:A3} also holds for $ \Ca_{\alpha,\Lambda}^\omega$. 
\end{theorem}

Our assumption that the set $ \Lambda $ is bounded away from the origin is rather severe. 
But, interestingly, removing this assumption would entail many extra subtleties, which we comment on at the end of the paper.

\bigskip 

We are inspired by the arithmetic ergodic theorems of Bourgain \cites{MR937582,MR1019960}. To explore the underlying complexity of these theorems, Bourgain studied the pointwise ergodic theorem formed from randomly selected subsets of the integers. In our notation, this lead to the study of maximal functions 
\begin{equation*}
\sup _{n} \Bigl\lvert 
\frac 1 {S_n} \sum _{m=1} ^{n }  X _{m} f (x-m)
\Bigr\rvert, 
\end{equation*}
for non-negative $ f\in \ell ^{p} (\mathbb Z )$.  This theme was studied by several authors 
\cites{MR1288356,MR3421994,MR1131794}, and we point in particular to the definitive results in the `lacunary' case \cite{MR1389623}.  
Our theorems are also closely related to the Wiener Wintner Theorem \cite{MR0004098}, which itself continues to have powerful and deep connections to ergodic theory \cites{MR958884,MR2246591} and harmonic analysis \cites{MR2420509,MR2881301}.  One can compare the results here with that of say \cite{14100806}, which obtains much stronger theorems, but for averages, as opposed to the singular sums of this paper.  

Our subject is also connected to the discrete harmonic analysis also inspired by Bourgain's arithmetic ergodic theorems, and promoted by Stein and Wainger \cites{MR1056560,MR1719802}.  This area remains quite active. Besides these older papers 
\cites{MR2188130,MR2318564}, the reader should also reference these very recent papers for interesting new developments \cites{MR2872554,151206918,160408695,151207524,151207523, 151207518}. 

The sparse bounds have been a recent and quite active topic in continuous harmonic analysis, 
see \cites{150105818,14094351,160305317} and references therein for a guide to this subject. Their appearance in the discrete settings is new. In particular, the weighted inequalities that are corollaries to our main theorem have very few precedents in the literature.

\smallskip 
The techniques of our proofs straddle (discrete) harmonic analysis and probability theory. We will use standard facts about maximal functions, and the Carleson theorem itself.  On the probability side, we reference standard large deviation inequalities for iid random variables, and martingales, to control random Fourier series.   
The method to obtain the sparse bounds is illustrated, in a simpler way, in the argument of \cite{lacey_spencer}, which also addresses random discrete inequalities.

\section{Preliminaries}

\subsection{Notation}

With $X_m$ as in \eqref{e:X},  we let $Y_m = X_m - \sigma_m$, and 
\[   W_m = \sum_{n=1}^m \sigma_m.\]
By the integral test, we note that $W_m = \frac{1}{1-\alpha} m^{1-\alpha} + O(1)$. 
 Moreover $ \textup{Var} (S_n) \leq W_n$.  And, it is well-known that, for any $\epsilon > 0$,
\begin{equation} \label{e:BC0}
|S_m - W_m| \lesssim m^{\epsilon + \frac{1- \alpha}{2}} 
\end{equation}

We will make use of the modified Vinogradov notation. We use $X \lesssim Y$, or $Y \gtrsim X$ to denote the estimate $X \leq CY$ for an absolute constant $C$. We use $X \approx Y$ as
shorthand for $Y \lesssim X \lesssim Y $.
We also make use of big-O notation: we let $O(Y)$ denote a quantity that is $\lesssim Y$. 

Since we will be concerned with establishing a priori $\ell^p(\Z)$-estimates in this paper, we will restrict every function considered to be a member of a ``nice'' dense subclass: each function on the integers will be assumed to have {finite support}.

\subsection{Fourier Transform}
As previously mentioned, we let $e(t):= e^{2\pi i t}$. 
The Fourier transform on $ f \in \ell ^2 (\mathbb Z )$ is defined by 
\begin{equation*}
\mathcal F  f (\beta ) = \sum_{n} f (n) e (- \beta n).  
\end{equation*}
This is a unitary map from $ \ell ^2 (\mathbb Z )$ to $ L ^{2} (\TT)$.  In particular, we have for convolution 
\begin{equation*}
\mathcal F  (f \ast g) = \mathcal F f \cdot \mathcal F g. 
\end{equation*}
In particular, all of our Theorems can be understood as maximal theorems over convolutions. It will be convenient to study the corresponding Fourier multipliers.  Indeed, the following technical   lemma exhibits  the way that  small Minkowski dimension is used. (It is so to speak a variant of Sobolev embedding, for sets of small Minkowski dimension.) 

\begin{lemma}[Lemma 2.4 of \cite{151206918}]\label{SOB}
Suppose   $\Lambda \subset [0,1]$ has upper Minkowski dimension at most $0 < d < 1$, as given in \eqref{e:Mink}
Suppose that $\{ T_\lambda : \lambda \in [0,1] \}$ is a family of operators so that for each $f \in \ell^2(\Z)$, $ T_\lambda f(x)$
is differentiable in $\lambda \in [0,1]$. Set
\begin{gather}\label{e:a}
 a := \sup_{\lambda \in [0,1]} \| T_\lambda \|_{\ell^2(\Z) \to \ell^2(\Z)}, 
\\  \label{e:A}
\textup{and} \qquad 
A := \sup_{\lambda \in [0,1]} \| \partial_\lambda T_\lambda \|_{\ell^2(\Z) \to \ell^2(\Z)}.
\end{gather}
Then we have the maximal inequality below 
\begin{equation}\label{e:SOB}
 \| \sup_\Lambda |T_\lambda f|\, \|_{\ell^2(\Z)} \lesssim C_d^{1/2} (a + a^{1-d/2} A^{d/2} ) \|f\|_{\ell^2(\Z)}. 
\end{equation}
\end{lemma}

In application, the quantities in \eqref{e:a} and \eqref{e:A} are estimated on the Fourier side.  
This will be used in settings where $ a \ll 1$ and $  1 \ll A \ll  a ^{-m}$, for some large integer $ m$. 
Then, for $ 0< d < 1/m$ sufficiently small, the right side of \eqref{e:SOB} will be small.  

\subsection{Sparse Operators}\label{s:sparse}

A \emph{sparse}   collection of intervals $ \mathcal S$  satisfy this essential condition: 
There is a collection of pairwise disjoint sets of the integers   $ \{E (I) \;:\; I\in \mathcal S\}$ so that $ \lvert  E (I)\rvert > \tfrac 1 {10}  \lvert  I\rvert $ for all $ I\in \mathcal S$.   A \emph{sparse bilinear form} is defined in terms of a choice of index $ 1\leq r < \infty $, and a sparse collection of intervals $ \mathcal S$.  Define 
\begin{equation}\label{e:SP_def}
\begin{split}
\Pi _{\mathcal S, r} (f,g) &= \sum_{I\in \mathcal S} \langle f \rangle_{I,r} \langle g \rangle_{I,r} \lvert  I\rvert 
\\
\langle f \rangle_{I,r}  & := \Bigl[ \lvert  I\rvert ^{-1} \sum_{n\in I} \lvert  f (n)\rvert  ^{r} \Bigr] ^{1/r}. 
\end{split}
\end{equation}
If the role of the sparse collection is not essential, it will be suppressed in the notation.

Sparse bounds are known for some operators $ T$, taking this form:   For all $ f, g$ finitely supported on $ \mathbb Z $, there is a choice of sparse operator $ \Pi _{r}$ so that 
\begin{equation}\label{e:Sr}
   \lvert  \langle Tf,g \rangle\rvert \lesssim \Pi _{r} (f,g),  
\end{equation}
where the implied constant is independent of $ f, g$.   We refer to this as \emph{the sparse property of index $ r$,} and write $ T \in \textup{Sparse}_r$

\begin{theorem}\label{t:sparse} We have these sparse bounds.  
\begin{enumerate}
\item For the maximal function $ M$, we have $ M \in \textup{Sparse}_1$.  
\item   For the Carleson operator $ C$ of  \eqref{e:Carleson}, we have for all $ 1< r < 2$, that $ C\in \textup{Sparse}_r$.  
\end{enumerate}
\end{theorem}

The bound for the maximal function is very easy, and not that sharp. The bound for the Carleson operator follows for instance from \cite{MR3484688}*{Theorem 4.6}.  
One of the fascinating things about the sparse bound is that they easily imply weighted inequalities.  

  For non-negative function $ w$, we define the Muckenhoupt $ A_p$ and \emph{reverse H\"older} characteristics by 
\begin{gather*}
[w] _{A_p} = \sup _{Q} \Bigl[\frac { w ^{\frac 1 {1-p}} (Q) } {\lvert  Q\rvert }\Bigr] ^{p-1} \frac {w (Q)} {\lvert  Q\rvert  } < \infty 
\\
[w] _{RH_p} = \sup _{Q} \frac { \langle w ^{p} \rangle_Q ^{1/p}} {\langle w \rangle_Q} < \infty 
\end{gather*}
Above, we are conflating $ w$ as a measure and a density, thus $w ^{\frac 1 {1-p}} (Q)  =\int _{Q} w (x) ^{\frac 1 {1-p}} \;dx $.  
And, we are stating the definition as if it on Euclidean space, but the theory transfers to the integers in a straight forward way. 
 We have these estimates, a corollary to \cite{MR3531367}*{Prop 6.4}.   

\begin{theorem}\label{t:wts}   For all $ 1< r < 2$,  $ r< p < r'$ and weights $ w$ there holds 
\begin{equation}  \label{e:wts}
\Pi _r (f, g) \leq  C ([w] _{A_ {p/r}}, [w ] _{RH _{r/ (r- p (r-1))}})
\lVert f\rVert _{ \ell ^{p} (w)} \lVert g\rVert _{ \ell  ^{p} (w)}. 
\end{equation}
\end{theorem}

The sharp bound for the constant on the right is computed in \cite{MR3531367}. There is little doubt that the results of this paper can be improved, so we don't track that constant.

\section{The Proof of Theorem \ref{random0}}
Theorem \ref{random0} concerns the maximal function in \eqref{e:Tdef}.   For fixed $ \lambda $ the summands in \eqref{e:Tdef} consist in part of  
\begin{align} \label{e:R0}
X_{|m|} \frac{e(\lambda m)}{S_{|m|}}  & =  c_\alpha \frac{e(\lambda m)}{|m|}  
\\  \label{e:R1} 
&\quad +   \Bigl[ \frac{ \sigma _ {\lvert  m\rvert }} {W_{|m|}} - \frac {c_\alpha} {\lvert  m\rvert } \Bigr]    {e(\lambda m)} 
\\  \label{e:R2}
&\quad + Y_ {\lvert  m\rvert }  \frac{e(\lambda m)}{W_{|m|}}  
\\ &  \label{e:R3}
\quad + 
X_{|m|}{e(\lambda m)} \Bigl[\frac 1 {S_{|m|}}  -  \frac 1 {W _{\lvert  m\rvert }}  \Bigr] 
\end{align}
This leads to the decomposition of the maximal operator in \eqref{e:Tdef}, upon multiplying each term by $\sgn(m)$. We will address them in  order, with the restriction on Minkowski dimension arising from only the term in \eqref{e:R2}.

The first and most significant term is associated with \eqref{e:R0}, which is entirely deterministic, and in fact the associated maximal function is exactly  the Carleson Theorem \ref{disc}, hence  we have the sparse bound from Theorem~\ref{t:sparse}.  The relevant sparse bound is  
$ C \in \textup{Sparse}_r$, for all $ 1< r < 2$.  
The second term \eqref{e:R1} is entirely trivial. As follows from \eqref{e:BC0}, we have almost surely 
\begin{align} \label{e:R1<}
\Bigl\lvert 
\frac { \sigma _m} {W _{ m }} -  \frac { c _{\alpha }} {m} 
\Bigr\rvert 
& = \frac {\lvert   m ^{1- \alpha } - c _{\alpha }W _{m}\rvert } { W_m \cdot m} \lesssim m ^{-2  }. 
\end{align}
Convolution with $ \frac 1 {m ^2 }$ is easily seen to satisfy a sparse bound, with $ r=1$.  

The third  term   \eqref{e:R2} is the one that imposes a condition on $ \Lambda $, the set that defines the maximal operator. 
We have this important  Lemma, which controls a relevant maximal function in $ \ell ^2 $-norm.   Define 
\begin{equation}\label{e:Pk}
P_k f := 
\sup _{\lambda \in \Lambda }  
\Bigl\lvert 
\sum_{m \;:\; 2 ^{k} \leq \lvert  m\rvert  < 2 ^{k+1}}
 Y_ { m }  \frac{e(\lambda m)}{W_{m}}   f (x-m) \Bigr\rvert.  
\end{equation}

\begin{lemma}\label{l:R2}   Assume that $ \Lambda $   has Minkowski dimension strictly less than $   {1 - \alpha } $. 
Then there is a positive choice of $ \eta = \eta (\alpha )>0$, so  for all integers $ k \in \mathbb N $, we have almost surely 
\begin{gather}  \label{e:R1<}
\sup _{k \in \mathbb N }  \lVert P_k \;:\; \ell ^1  \to \ell ^1 \rVert +  \lVert P_k \;:\; \ell ^ \infty   \to \ell ^ \infty  \rVert < \infty , 
\\
\label{e:R2<}
\sup _{k \in \mathbb N } 2 ^{\eta k} \lVert P_k \;:\; \ell ^2 \to \ell ^2 \rVert < \infty . 
\end{gather}
\end{lemma}

\begin{proof}
The first claim is a consequence of the Strong Law of Large Numbers. Note that 
\begin{equation*}
 \lVert P_k \;:\; \ell ^1  \to \ell ^1 \rVert  
 \lesssim 
 \sum_{m \;:\; 2 ^{k} \leq \lvert  m\rvert  < 2 ^{k+1}}
  \frac{ \lvert  Y_ { m } \rvert }{W_{m}}  . 
\end{equation*}
And, the latter is uniformly bounded almost surely.

The $ \ell ^2 $ estimate, which has a gain in operator norm,   is a consequence of Lemma \ref{SOB},  which bounds supremums like those in \eqref{e:R2<} in terms of the $ \ell ^{\infty }$ norm of the multipliers, and the derivatives of the multipliers.   Due to the form of the sums, the multipliers are translations  by $ \lambda \in \Lambda $ of the functions of $  \theta $  below.  
Now, Lemma \ref{SOB} requires two estimates, the first is the $ L ^{\infty } (d \theta )$ estimate, for which we have 
\begin{equation}\label{e:R2A}
\mathbb P 
\Bigl(
\Bigl\lVert 
\sum_{m \;:\; 2 ^{k} \leq  \lvert  m\rvert  < 2 ^{k+1}}
 Y_ { m }  \frac{e(\theta m)}{W_{m}}   
\Bigr\rVert _{\infty } > C\sqrt  k \cdot 2 ^{ k (  \alpha -1)/2} 
\Bigr) \leq 2 ^{-k}, 
\end{equation}
for appropriate constant $ C$.  (We will recall a  proof in  Lemma \ref{l:C3} below.) 
We also need an estimate for the derivative in $ \theta $ of the functions above, which is clearly of the form 
\begin{equation}\label{e:R2B}
\mathbb P 
\Bigl(
\Bigl\lVert 
\sum_{m \;:\; 2 ^{k} \leq  \lvert  m\rvert  < 2 ^{k+1}}
 m Y_ { m }  \frac{e(\theta m)}{W_{m}}   
\Bigr\rVert _{\infty }  > C  \sqrt k \cdot  2 ^{ k (  \alpha +1)/2}  
\Bigr) \leq  2 ^{-k}
\end{equation}
By the Borel-Cantelli Lemma, we see that the union of these two events occur finitely often, almost surely. 

Apply \eqref{e:SOB}, with $ a =  \sqrt  k \cdot 2 ^{ k (  \alpha -1)/2} $ and $ A = \sqrt k \cdot  2 ^{ k (  \alpha +1)/2} $. 
We  see that the conclusion of the Lemma holds provided 
\begin{align*}
 (\alpha -1) (1- d/2)  + (\alpha +1)d /2   < 0
\end{align*}
where $ \alpha $ is the constant associated to the selector random variables, and $ d$ is the Minkowski dimension of $ \Lambda $. 
This inequality is true for $ d <   1- \alpha $. This completes the proof.  

\end{proof}

The previous Lemma implies the $ \ell ^{p}$-controll of the term  associated to \eqref{e:R2}, after a straight forward interpolation between \eqref{e:R1<} and \eqref{e:R2<}.  We turn to the sparse bound. 

\begin{corollary}\label{c:Pk} There is a $ \eta >0$, so that almost surely, there is a finite constant $ C _{\omega } >0$ so that  we have for all integers $ k$, intervals $ I$ of length $ 2 ^{k}$, 
and  functions $ f, g$  supported on $ I$,  
\begin{equation} \label{e:cPk}
\lvert  \langle  P_k f , g \rangle \rvert \leq C _{\omega }
2 ^{ - \eta 'k}
 \langle f \rangle _{3I,r} \langle g \rangle _{I, r} \lvert  I\rvert    \qquad  0 < 2 - r < c (\eta , \delta ). 
\end{equation}
The constants  $  c (\eta , \delta )$ and $ \eta ' = \eta ' (\eta , \delta ,r  )$  are positive, sufficiently small constants. 
\end{corollary}

On the right above, we have a geometric decay in $ k$, and a sum over intervals of fixed length which are disjoint. 
It is easy to see that with $ f$ and $ g$ fixed, we have 
\begin{equation*}
\sum_{k}2 ^{- \eta'k} \sum_{I\in \mathcal D \;:\; \lvert  I\rvert = 2 ^{k} }  \langle f \rangle _{3I,r} \langle g \rangle _{I,r}  \lvert  I\rvert 
\lesssim \Pi _{r} (f,g), 
\end{equation*}
for appropriate sparse operator $ \Pi $. 
This completes the sparse bound for the term associated to \eqref{e:R2}.

\begin{proof}
Observe that almost surely, there exists $C_\omega < \infty$, so that these inequalities hold. 
\begin{equation} \label{e:CPK}
\lvert  \langle  P_k f , g \rangle \rvert \leq C _{\omega } 
\begin{cases}
2 ^{- \eta k }   \langle f \rangle_ { 3I, 2} \langle g   \rangle _{I, 2} \lvert  I\rvert 
\\
2 ^{\alpha k} \langle f \rangle_ { 3I, 1} \langle g   \rangle _{I, 1}  \lvert  I\rvert 
\end{cases}
\end{equation}
The top line follows from \eqref{e:R2<}. 
The last line is the inequality that is \emph{below duality.} (And, geometric growth in $ k$.)  It follows from estimate 
\begin{equation*}
\lvert  \langle  P_k f , g \rangle  \rvert 
\leq 
\sum_{x} \sum_{ n \;:\; 2 ^{k}\leq \lvert n  \rvert < 2 ^{k+1} } \frac {\lvert  Y_n\rvert } { \lvert  n\rvert ^{1- \alpha} } 
\lvert  f (x - n)\rvert  \cdot \lvert  g (x)\rvert. 
\end{equation*}
Then, use the $ \ell ^{1}$-norm on $ f$, the same on $ g$, and the $ \ell ^{\infty }$ norm on $ \lvert  Y_n\rvert $. 

To conclude \eqref{e:cPk}, interpolate between the top and bottom estimates in \eqref{e:CPK}. 
The bottom line has a fixed positive geometric growth, while the $ \ell ^2 $ estimate has a fixed negative geometric growth. For a choice of $ 1< r <2$, with $ 2-r$ sufficiently small, we will have the geometric decay claimed.  
\end{proof}

The control of the fourth term associated with \eqref{e:R3},  again requires no cancellation, as 
follows immediately from this next Lemma.  

\begin{lemma}\label{l:R3}  Almost surely, we have 
\begin{equation}\label{e:R3<}
\sum_{m\neq 0}  
\Bigl\lvert 
\frac 1 {S_{|m|}}  -  \frac 1 {W _{\lvert  m\rvert }} 
\Bigr\rvert \cdot  \lvert f (x - m) \rvert  \in \textup{Sparse}_1 .  
\end{equation}

\end{lemma}

\begin{proof} 
We only discuss the case of $ m>0$. By the Law of the Iterated Logarithm, we have 
\begin{equation*}
 S_ {m } = W _{m} +  O ( m ^{ \frac { 1- \alpha }2} \sqrt {  \log \log m}). 
\end{equation*}
And, recall that $ W_m \sim m ^{ (1- \alpha )}$. It follows that 
\begin{equation*}
\Bigl\lvert 
\frac 1 {S_{ m}}  -  \frac 1 {W _{ m}} 
\Bigr\rvert \lesssim \frac {\sqrt {\log\log m}}  { m ^{ \frac 32 - \frac \alpha 2}} \lesssim m ^{- \beta }
\end{equation*}
where $ \beta >1$.  The sparse bound is then immediate.  
\end{proof}

\section{Proof of Theorem \ref{random}}
The summands  in \eqref{e:C} are rewritten as below, in which we assume that $ m>0$. 
\begin{align}
X_m e(\lambda m) \cdot & \frac{f(x-S_m) - f(x+S_m)}{S_m} 
\\ \label{e:C1}
= & X_m e(\lambda m) {(f(x-S_m) - f(x+S_m))}  \Bigl\{ \frac1{S_m} - \frac 1 {W_m}  \Bigr\}
\\
&  \label{e:C2}\quad + 
\Bigr\{     \frac {\sigma _m } {W_m} - \frac {c_ \alpha } m \Bigr\} e(\lambda m) \cdot {(f(x-S_{m-1}-1) - f(x+S_{m-1}+1))} 
\\  \label{e:C3}
     & \quad +      Y _m  e(\lambda m) \cdot \frac{f(x-S_{m-1}-1) - f(x+S_{m-1}+1)}{W_m} . 
\\ &\quad + 
  {c_ \alpha }    e(\lambda m) \cdot \frac {f(x-S_{m-1}-1) - f(x+S_{m-1}+1)} m  . 
      \label{e:C4}
\end{align}
In the first stage we simply replace $ \frac 1 {S_m} $ by $ \frac 1 {W_m}$. But the remaining terms  use the identity 
\begin{equation*}
X_m f (x+ S_m) = X_m f(x+S_{m-1} +1 ) , 
\end{equation*}
which step is  motivated by  a martingale argument below.  

The term associated with \eqref{e:C1} is  controlled by the estimate \eqref{e:R3<}, and that for \eqref{e:C2} is  entirely similar.  The term in \eqref{e:C3} is analogous to Corollary \ref{c:Pk}, for which we need this Lemma. 
Define the maximal operator 
\begin{equation}\label{e:Qk}
Q_k f := 
\sup _{\lambda \in \Lambda } 
\Bigl\lvert 
\sum_{ m \;:\; 2 ^{k} \leq  \lvert  m\rvert  \leq 2 ^{k+1}}     Y _m  e(\lambda m) \cdot \frac{f(x-S_{m-1}-1) - f(x+S_{m-1}+1)} {W_m} 
\Bigr\rvert
\end{equation}

\begin{lemma}\label{l:C3} 
Assume that $ \Lambda \subset \mathbb T $ has Minkowski dimension at most $ 1- \alpha $, then there is a $ \eta >0$ so 
that we have almost surely 
\begin{gather} \label{e:Qk1}
\sup _{k}  \lVert Q_k \;:\; \ell ^1 \to \ell ^1 \rVert +  \lVert Q_k \;:\; \ell ^ \infty  \to \ell ^ \infty  \rVert  < \infty , 
\\
\label{e:Qk2}
\sup _{k} 2 ^{\eta k} \lVert Q_k \;:\; \ell ^2 \to \ell ^2 \rVert < \infty . 
\end{gather}
\end{lemma}

\begin{proof}
The top line follows from the Strong Law of Large Numbers.  We turn to the second line, where there is geometric decay. 
There are no cancellative effects between positive and negative translations, and so we only consider the positive ones. 
The $ \ell ^2 $ bound is a  consequence of Lemma \ref{SOB}, and so we need to consider the multipliers 
\begin{equation*}
M (\lambda , \theta  ) := 
\sum_{ m \;:\; 2 ^{k} \leq m \leq 2 ^{k+1}}     Y _m  \frac {e(\lambda m + \theta( S_ {m-1} +1))  } {W_m}  . 
\end{equation*}
So to prove the Lemma, it suffices to show that with probability at least $ 1- 2 ^{ - \epsilon k}$, we have the two inequalities 
\begin{gather}\label{e:C3A}
\lVert 
M (\lambda , \theta  )
\rVert _{L ^{\infty } (\lambda , \theta  )} \lesssim 2 ^{k (-(\alpha + 1)/2  +\epsilon )}, 
\\ \label{e:C3B}
\lVert  \partial _{\lambda  }M (\lambda , \theta  )
\rVert _{L ^{\infty } ( \lambda , \theta  )} \lesssim 2 ^{k ( (1-\alpha)/2   +\epsilon )} . 
\end{gather}

The summands  in the definition of $ M (\lambda , \beta )$, for fixed $ \lambda $ and $ \beta $, form a bounded martingale difference sequence, with square function bounded in $ L ^{\infty } (\Omega )$ by 
\begin{equation*}
\Bigl[
\sum_{ m \;:\; 2 ^{k} \leq m \leq 2 ^{k+1}}     \frac {\sigma _m  } {W_m ^2 } 
\Bigr] ^{1/2} \lesssim 2 ^{- k \frac {1-\alpha }2} .  
\end{equation*}
It is well-known that such martingale differences are sub-gaussian,  
hence,  uniformly in $ \lambda $ and $ \beta $, we have 
\begin{equation} \label{e:azuma}
\mathbb P (  \lvert  M (\lambda , \beta )\rvert > 2 ^{k (-\frac {1-\alpha  }2  +\epsilon )}) 
\lesssim \operatorname {exp}(  -  2 ^{2 \epsilon k}). 
\end{equation}
But, $ M (\lambda , \beta )$ clearly has gradient at most $ 2 ^{k}$ in norm. That means to test the $ L ^{\infty } (\lambda , \beta )$ norm, we apply the inequality \eqref{e:azuma} on a set of at most $ 2 ^{2k}$ choices of $ (\lambda , \beta )$. Therefore \eqref{e:C3A} clearly follows. The analysis for \eqref{e:C3B} is similar.  

\end{proof}

 With this bound in hand, we can repeat the proof of Corollary \ref{c:Pk}, and conclude that almost surely we have 
 \begin{gather*}
 \sum_{k}  \lVert Q_k \;:\; \ell ^{p} \to \ell ^{p}\rVert < \infty , \qquad 1< p < \infty ,
 \\
\sum_{k}  \lvert  \langle Q_k f, g \rangle\rvert \in \textup{Sparse}_ {r,r} (f,g)  \qquad     0< r < 2 - c (d, \delta ).  
\end{gather*}
This completes the analysis of the term associated with \eqref{e:C3}.  

\medskip

The term associated with \eqref{e:C4} is  arithmetic in nature.   Let $ a_j$ be the smallest positive integer such that $ S _{a_j}=j$.  It is a consequence of the Strong Law of Large Numbers that we have 
\begin{equation}\label{e:aj}
 a_j = p_j +  O( j^{\epsilon + \frac{1}{2(1-\alpha)}})=  \lfloor C_\alpha j^{\frac{1}{1-\alpha}} \rfloor   + O( j^{\epsilon + \frac{1}{2(1-\alpha)}}), 
\end{equation}
where $0 < C_\alpha = (1-\alpha)^{\frac{1}{1-\alpha}} < \infty$.  Now observe that for fixed $ \lambda $, we have 
\begin{align}
\sum_{m >1}  e(\lambda m) \cdot &\frac {f(x-S_{m-1}-1) - f(x+S_{m-1}+1)} m  
\\   \label{e:sumAj}
& = \sum_{j=1}^\infty A _{j} (\lambda ) (f(x-j) - f(x+j)), 
\\ \label{e:Aj}
\textup{where}  \quad A_j (\lambda )  &:= \sum_{ m= a_{j-1}}^{a_j -1}  \frac{e(\lambda m)}{m}, 
\end{align}
and  $a_0 := 1$ if $a_1 > 1$.

Attention turns to the coefficients $ A _{j} (\lambda )$.   The point below is that if $ j$ is small relative to $ \lambda $, we have an excellent approximation to $ A_j (\lambda )$, and otherwise, the coefficient is small for other reasons. 

\begin{lemma}\label{l:Aj}  The these two inequalities below holds uniformly over all compactly supported functions $ f$, almost surely.   
\begin{gather}\label{e:Aj<}
\sup _{0< \lambda <1}  \sum _{ \lvert  j\rvert < \lambda ^{- \frac \alpha {1- \alpha }}} \lvert   \Delta _j f (x-j) \rvert  
\lesssim M f (x)
\\
\textup{where} \qquad 
\Delta _j =   A_ j (\lambda )  - \frac {e ( p _{j} \lambda )} j , 
\\ \label{e:Aj<<}
 \sup _{0 < \lambda < 1} \sum_{ \lvert  j\rvert  \geq  \lambda ^{- \frac \alpha {1- \alpha }} }  \lvert  A _{j} (\lambda )  f (x-j)\rvert 
\lesssim M f (x). 
 \end{gather}
\end{lemma}

Above, $ M$ is the maximal function, and it is in $ \textup{Sparse}_1$.

\begin{proof} 
We begin with an elementary estimate.  Set 
\begin{equation}\label{e:rj}
 r_j = p _{j} - p _{j-1} \simeq j ^{\frac 1 {1- \alpha }-1} = j ^{\frac \alpha {1- \alpha }}. 
\end{equation}
We use this notation to rewrite $ A_j (\lambda )$ in terms of the Dirichlet kernel.  
\begin{align}
A _{j} (\lambda )  & = e ( p _{j} \lambda ) \sum_{m=0} ^{r _{j}-1} \frac {e (-\lambda m) } { p _{j-1}+m} + O ( j ^{- \frac 1 {1- \alpha }}) 
\\
&= \frac {e ( p _{j} \lambda )} {p _{j-1}} \sum_{m=1} ^{r _{j}}   {e (-\lambda m) } + O \Bigl(\tfrac {r_ {j-1}} { p _{j-1} ^2 } \Bigr) 
\\
& =  \frac {e ( p _{j} \lambda )} {p _{j-1}} D _{r_j} (- \lambda )  + O (j ^{-1 - \frac 1 {1- \alpha }}) . 
\end{align}
In the last line, $ D _n $ denotes the $ n$th Dirichlet kernel.  Clearly, convolution with the Big-Oh term is bounded 
by the maximal function, so that we continue with the term involving the Dirichlet kernel.

By the estimate $ \lvert  D _{n} (\lambda )- n \rvert \lesssim n\lambda   $, for $ 0 < \lambda < 1$, we have 
\begin{align*}
d_j =\Bigl\lvert 
A _{j} (\lambda ) - \frac {{e ( p _{j} \lambda ) r _{j}}} {p _{j-1}} 
\Bigr\rvert \lesssim r_j \lambda  \lesssim  j ^{ \frac \alpha {1- \alpha }} \lambda .  
\end{align*}
It follows that for non-negative $ f$, 
\begin{equation*}
\sup _{0 < \lambda < 1} \sum_{j \;:\; 0\leq  \lvert  j\rvert  \leq   \lambda ^{- \frac \alpha {1- \alpha }} }   d_j   f (x-j) 
\lesssim  M f (x).  
\end{equation*}
This is nearly completes the proof of \eqref{e:Aj<}. The last step is to observe that 
\begin{align*}
\Bigl\lvert 
\frac { r _{j}} {p _{j-1}}  - \frac 1  {j} 
\Bigr\rvert  \lesssim j ^{-2}. 
\end{align*}
And, convolution with respect to $ j ^{-2}$ is bounded by the maximal function as well. This completes the proof of \eqref{e:Aj<}.  

\smallskip 

By the estimate $ \lvert  D _{n} (\lambda )\rvert \lesssim \lambda ^{-1}  $, for $ 0 < \lambda <1 $, we have 
\begin{equation}\label{e:AjFar}
\lvert  A _{j} (\lambda )\rvert \lesssim ( p_j \lambda )^{-1}  \simeq    j ^{- \frac 1 {1- \alpha }} \lambda .   
\end{equation}
Hence, we have, for non-negative $ f$, 
\begin{equation*}
\sup _{0 < \lambda < 1} \sum_{ \lvert  j\rvert  >  \lambda ^{- \frac \alpha {1- \alpha }} }  {\lvert  A _{j} (\lambda )\rvert }  f (x-j) 
\lesssim  M f (x), 
\end{equation*}
where $ M$ denotes the discrete Hardy-Littlewood maximal function.  This proves \eqref{e:Aj<<}.

\end{proof}

We see that the proof of our Theorem is reduced to this deterministic, and trivial, result.  
Here, we simply use a crude bound, and the assumption that $ \Lambda $ has no points close to the origin. 
With the summation condition, this bound is trivial.  

\begin{lemma}\label{l:deterministic}    For $ 0 < \epsilon < \tfrac 14$
\begin{equation}\label{e:deterministic}
\sup _{ \epsilon < \lvert  \lambda \rvert < \tfrac 12 } 
\Bigl\lvert 
\sum _{1 < j < \lambda ^{- \frac \alpha {1- \alpha }}} \frac {e (\lambda p_j)} {j} (f (x+j) - f (x-j))
\Bigl\rvert 
  \lesssim  (\log 1/ \epsilon ) M f (x) 
\end{equation}
\end{lemma}

This last Lemma is the one that uses the assumption in Theorem~\ref{random} that $ \Lambda $ is bounded away from the origin. 
If we remove this assumption, the Lemma above shows that  arithmetic issues become paramount. 
Indeed, we see that the main results of \cites{151206918,160408695} are relevant. But, here we note that 
(a) the sparse variants of the main results in these papers are not known, 
(b) these are very involved papers,  
and 
(c) their main results would have to be extended.  
In particular, \cite{151206918} would have to be extended to the case of an arbitrary monomial in the oscillatory term, as well as incorporating maximal truncations into the main theorem.  
(We hope to address these in a future paper.)
In discrete harmonic analysis, randomly formed operators typically do not inherit any difficult arithmetic structure. 
It is notable in these questions that they can.

\end{document}